\documentclass[11pt,a4paper]{article} 
\usepackage{amsmath,amssymb,amsthm,amsfonts}
\usepackage{enumerate,color,bm}
\usepackage[utf8]{inputenc}
\usepackage[T1]{fontenc}
\usepackage{geometry}
\numberwithin{equation}{section} 
\newtheorem{thm}{Theorem}[section]

\newtheorem{lem}[thm]{Lemma}

\theoremstyle{definition}

\newcommand{\eps}{\varepsilon}

\newcommand{\RN}{\mathbb{R}^N}

\newcommand{\R}{\mathbb{R}}
\newcommand{\f}[2]{\frac{#1}{#2}}
\newcommand{\na}{\nabla}

\newcommand{\norm}[2][]{\left\|#2\right\|_{#1}}
\newcommand{\Lom}[1]{L^{#1}(\Omega)}
\newcommand{\kl}[1]{\left(#1\right)}
\newcommand{\Ombar}{\overline{\Omega}}
\newcommand{\Om}{\Omega}

\newcommand{\tmax}{T_{\rm max}}
\newcommand{\lp}[2]{\|#2\|_{L^{#1}(\Omega)}}

\newcommand{\ol}{\overline}
\newcommand{\into}{\int_\Omega}
\newcommand{\HP}{\mathcal{P}}

\newcommand{\io}{\int_{\Omega}}
\newcommand{\nn}{\nonumber}
\newcommand{\calP}{\mathcal{P}}

\begin{document}
\begin{center}
    \LARGE{{\bf 
 Singular sensitivity in a Keller--Segel--fluid system    }}
\end{center}
\vspace{5pt}
\begin{center}
    Tobias Black\\
    \vspace{2pt}
    Universit\"at Paderborn, 
    Institut f\"ur Mathematik,\\ 
    Warburger Str.\ 100, 33098 Paderborn, Germany\\
    {\tt tblack@math.uni-paderborn.de}\\
    \vspace{12pt} 
    Johannes Lankeit\\
    \vspace{2pt}
    Universit\"at Paderborn, 
    Institut f\"ur Mathematik,\\ 
    Warburger Str.\ 100, 33098 Paderborn, Germany\\
    {\tt jlankeit@math.uni-paderborn.de}\\
    \vspace{12pt}
    Masaaki Mizukami\\
    \vspace{2pt}
    Department of Mathematics, 
    Tokyo University of Science\\
    1-3, Kagurazaka, Shinjuku-ku, Tokyo 162-8601, Japan\\
    {\tt masaaki.mizukami.math@gmail.com}\\
    \vspace{2pt}
\end{center}
\begin{center}    
    \small \today
\end{center}

\vspace{2pt}
\newenvironment{summary}
{\vspace{.5\baselineskip}\begin{list}{}{%
     \setlength{\baselineskip}{0.85\baselineskip}
     \setlength{\topsep}{0pt}
     \setlength{\leftmargin}{12mm}
     \setlength{\rightmargin}{12mm}
     \setlength{\listparindent}{0mm}
     \setlength{\itemindent}{\listparindent}
     \setlength{\parsep}{0pt}
     \item\relax}}{\end{list}\vspace{.5\baselineskip}}
\begin{summary}
{{\bf Abstract.}
In bounded smooth domains $\Omega\subset\mathbb{R}^N$, $N\in\{2,3\}$ considering the chemotaxis--fluid system 
\begin{align*}
 n_t + u\cdot \nabla n &= \Delta n - \chi \nabla \cdot(\frac{n}{c}\nabla c)\\
 c_t + u\cdot \nabla c &= \Delta c - c + n\\
 u_t + \kappa (u\cdot \nabla) u &= \Delta u + \nabla P + n\nabla \phi
\end{align*}
with singular sensitivity, we prove global existence of classical solutions for given $\phi\in C^2(\overline{\Omega})$, for $\kappa=0$ (Stokes-fluid) if $N=3$ and $\kappa\in\{0,1\}$ (Stokes- or Navier--Stokes fluid) if $N=2$ and under the condition that 
\[
 0<\chi<\sqrt{\frac{2}{N}}.
\]
\vspace*{1cm}
}
\end{summary}
\vspace{10pt}

\newpage
%
%
\section{Introduction}
\subsection{Chemotaxis--fluid models}
One of the assumptions underlying many mathematical models dealing with chemotaxis (i.e. the partially directed movement of cells in the presence of a chemical signal substance) is that -- apart from response to or production of the signal -- interaction with the environment can be neglected. For transport effects in liquid surroundings this may be justified in the case of single bacteria, but should no longer be assumed in presence of a larger number of cells, 
as experimentally shown in \cite{dombrowskietal}. Accordingly, inspired by the model suggested in \cite{tuval} and starting from the construction of weak solutions in \cite{Lorz,Liu_Lorz}, over the past few years, the analysis of chemotaxis--fluid models has begun to flourish, in which cells and signal substance are assumed to be transported by a fluid, whose motion is driven by gravitational forces induced by density differences between bacteria and the fluid, and whose velocity $u$ and pressure $P$ hence obey 
\[
         u_t + \kappa (u\cdot \nabla) u
        = \Delta u + \nabla P 
        + n \nabla\phi, 
        \qquad \nabla\cdot u = 0, 
\]
where $\phi\in C^2(\Omega)$ is a given gravitational potential and the parameter $\kappa\in\{0,1\}$ distinguishes between Stokes- and Navier--Stokes--governed fluid motion. The equations describing the time evolution of the population density $n$ and signal concentration $c$ differ from their more classical counterparts by a transport term only and are, in a general form, given by
\[
     \begin{cases}
         n_t + u\cdot\nabla n 
         = \Delta n - \nabla\cdot(n \chi(n,c) \nabla c), 
 \\[1mm]
         c_t + u\cdot\nabla c 
         = \Delta c + g(n,c). 
\end{cases} 
\]
Bacteria of the species \textit{Bacillus subtilis} that was used in the experiments in \cite{dombrowskietal,tuval} chemotactically respond to oxygen which they consume. A prototypical and popular choice of the functions in the above system is accordingly given by $g(n,c)=-nc$ and the classical $\chi(n,c)=\chi=\mathrm{const}$. This choice was, for convex domains, first covered by the existence results in \cite{win_CTNS_global_largedata}, also extension to nonconvex domains (\cite{jiang_wu_zheng}), results on convergence of solutions (\cite{win_arma}) and its rate (\cite{zhang_li_decay}) are available. The most advanced developments dealing with the full chemotaxis-\textit{Navier--Stokes} system in three-dimensional domains are constituted by the 
construction of global weak solutions in \cite{win_globalweaksolns} and the proof of their eventual regularization and convergence in \cite{win_transAMS}.
 Also model variants involving logistic source terms (\cite{braukhoff,lankeit_m3as}), several species (\cite{HKMY,CKM}), rotational sensitivity functions (\cite{win_gensolnstotensorvaluedsensitivities,caolan16,liu_wang}) and/or nonlinear diffusion (\cite{zhang_li,sachiko_positiondep}) have been studied and the interested reader can find additional information and references there or in \cite[Section 4.1]{BBTW}.

Chemotaxis-fluid systems describing a signal being produced by the cells themselves, as with the Keller--Segel type choice $g(n,c)=+n-c$, up to now have been studied in much fewer works. Global solutions were found to exist in 
a whole-space setting in the sense of mild solutions \cite{kozono_miura_sugiyama}, 
in systems with sensitivity functions that obey an estimate of the form $\chi(n,c)\le (1+n)^{-\alpha}$ for some $\alpha>0$ (\cite{winkler_wang_xiang,wang_xiang_16}), in the presence of additional logistic source terms (\cite{tao_win}), nonlinear diffusion (\cite{xie_wang_xiang}) or for sublinear signal production, that is, $g$ generalizing $g(n,c)=n^{\alpha}-c$ for some $\alpha\in(0,1)$ (\cite{black_sublinear}). 

\subsection{Singular sensitivity}
Another important class of chemotaxis models is formed by those with a singular sensitivity function, like $\chi(n,c)=\frac{\chi}{c}$ ($\chi$ being a positive constant). This form is suggested by the Weber--Fechner law of stimulus perception (see \cite{KS71_traveling}) and 
supported by experimental (\cite{kalinin_jiang_tu_wu}) and theoretical (\cite{xue}) evidence. Its characteristics are shaped by the chemotactic effects becoming very strong at fast-varying small signal concentrations -- and, indeed, for sufficiently large values of the coefficient $\chi$ (namely, $\chi>\frac{2N}{N-2}$), radial solutions undergoing blow-up within finite time have been found in the corresponding parabolic-elliptic fluid-free setting (\cite{nagai_senba}). On the other hand, for small values of $\chi$ and in the absence of fluid, classical solutions are known to exist globally in bounded domains of dimension two (\cite{biler99,nagai_senba_yoshida_ge}) or arbitrary dimension (\cite{win_singular}) and to be bounded (\cite{fujie}; for a generalization involving different diffusion coefficients see also \cite{zhao_zheng}), where the precise condition $\chi<\sqrt{\frac2N}$ imposed on the chemotactic coefficient in these works is known to be not strict: For two-dimensional domains, global existence 
and 
boundedness of classical solutions were shown for any slightly larger $\chi$ in \cite{lankeit_m2as}. There is still a range of values for $\chi$ where it is unknown whether blow-up can occur. Attempts at gaining insight here include the  consideration of system variants where either component is assumed to diffuse slowly (though not infinitely slowly) if compared to the other (\cite{fujie_senba_fastv_2d} and \cite{fujie_slowsignaldiffusion}) and the constructions of solutions within a weaker framework (\cite{win_singular,stinner_win,lankeit_winkler}) that at least cannot undergo blow-up in form of a persistent delta-type singularity.  

\subsection{The combination of fluid and singular sensitivity. Results of this work}
The study of chemotaxis systems accounting for both singular sensitivity functions and fluid has only just begun. In a signal-consumptive setting, the existence of weak solutions (\cite{wang_singular}) and their eventual smoothness (\cite{black_evsmooth}) have been shown as well as the existence of global classical  solutions under appropriate smallness conditions on the initial data (\cite{black_evsmooth}). 

To our knowledge, however, the corresponding system with production of the signal substance has not been treated yet and we wish to initiate these investigations with the present work and consider 
 \begin{equation}\label{cp}
     \begin{cases}
         n_t + u\cdot\nabla n 
         = \Delta n - \chi \nabla\cdot(\frac{n}{c}\nabla c), 
 \\[1mm]
         c_t + u\cdot\nabla c 
         = \Delta c - c + n, 
 \\[1mm]
        u_t + \kappa (u\cdot \nabla) u
        = \Delta u + \nabla P 
        + n \nabla\phi, 
        \quad \nabla\cdot u = 0 
        &\quad \text{in } \Omega\times(0,T)
     \end{cases}
 \end{equation}
in a bounded domain $\Omega\subset \R^N$, $N\in\{2,3\}$, with smooth boundary, and on the time interval $(0,T)$, $T\in(0,\infty]$, 
supplemented with the usual boundary conditions 
\begin{equation}\label{boundary}
         \partial_\nu n =
        \partial_\nu c = 0, \quad 
        u = 0 
        \qquad \text{in }\partial\Omega\times(0,T)
\end{equation}
and inital data
\begin{equation}\label{init}
         n(\cdot,0)=n_0,\ c(\cdot,0)=c_0,\ 
        u(\cdot,0)=u_0 
        \qquad \text{in } \Omega, 
\end{equation}
that satisfy 
 \begin{align}\label{condi;ini1}
   &0 \le n_{0} 
   \in C^0(\overline{\Omega}),
  \\ 
   &c_0 \in W^{1,\vartheta}(\Omega), 
   \quad \inf_{x\in \Omega}c_0(x) > 0,  
  \\ 
   &u_0 \in D(A^{\alpha}) 
 \end{align}
for some $\vartheta > N$, $\alpha \in \left(\frac{N}{4}, 1\right)$, with $A:=-\HP\Delta$ denoting the Stokes operator in $L^2_\sigma\!\left(\Omega\right)$ under homogeneous Dirichlet boundary conditions. Moreover, we will assume 
\begin{equation}\label{condi;ini2}
 \phi \in C^{2}(\overline{\Omega}). 
\end{equation}

We shall ask ourselves the question to what extent results and methods of the fluid-free case can be transferred to the present, more complex situation and will, indeed, recover the result on global existence for the same range of parameters as known from \cite{win_singular} for the fluid-free case: 

\begin{thm}\label{mainthm}
For $N\in\{2,3\}$ let $\Omega\subset \mathbb{R}^N$ 
 be a bounded domain with smooth boundary. 
 Suppose that $n_{0},c_0,u_0,\phi$ fulfil 
 \eqref{condi;ini1}{\rm --}\eqref{condi;ini2}
 and $\chi > 0$ satisfies
 \begin{align*}
 \chi < \sqrt{\frac 2N}. 
 \end{align*}
 Moreover assume that $\kappa\in\{0,1\}$ if $N=2$ and that $\kappa=0$ if $N=3$. 
 Then there exist functions 
    \begin{align*}
    n &\in C^0(\overline{\Omega}\times[0, \infty)) \cap 
        C^{2, 1}(\overline{\Omega}\times(0, \infty)), 
    \\
    c &\in C^0(\overline{\Omega}\times[0, \infty)) \cap 
        C^{2, 1}(\overline{\Omega}\times(0, \infty))
        \cap 
        L^\infty([0,\infty);W^{1,\vartheta}(\Omega)), 
    \\
    u &\in C^0(\overline{\Omega}\times[0, \infty)) \cap 
        C^{2, 1}(\overline{\Omega}\times(0, \infty)), 
    \\
    P &\in C^{1, 0}(\overline{\Omega}\times(0, \infty)) 
    \end{align*}
 which solve \eqref{cp}{\rm --}\eqref{init} 
 classically 
 in $\Omega\times [0,\infty)$.
 Moreover, the solution $(n, c, u,P)$ of 
 \eqref{cp}{\rm --}\eqref{init} 
 is unique, up to addition of 
 spatially constant functions to $P$. 
\end{thm}

\subsection{Technical challenges and plan of the paper}
While the main idea of deriving (local-in-time-)boundedness from the functional 
\[
 \io n^pc^{-r} 
\]
for suitable values of $p$ and $r$ (employed in \cite{win_singular} as well as \cite{fujie}) remains applicable (see Lemma \ref{lem;intnpc-r}), the presence of the transport terms poses obstacles in several respects: Firstly, the starting point for the iteration procedure underlying existence and boundedness proofs in \cite{win_singular} and \cite{fujie}, respectively, was to consider the equation for $c$ as inhomogeneous heat equation and to use the apparent bound on $\norm[\Lom1]{n(\cdot,t)}$ together with semigroup estimates. Now, however, the source term is not $n$ anymore, but $n-u\cdot\na c$, and a priori little is known about bounds for $u$ or even $\na c$, so that the reasoning of \cite[Lemma 2.4]{fujie} or \cite[Lemma 2.4]{win_singular} cannot be used here. We will hence resort to a differential inequality for $\io c^q$ for $q\in[1,\infty)$ (Lemma \ref{lem;id;DIforcq}) and estimate the production term arising therein by means of 
\[
 \io nc^{q-1} \le \kl{\io n^pc^{-r}}^{\f1p} \kl{\io c^{\f{pq-p+r}{p-1}}}^{\f{p-1}p}, 
\]
where the first factor can be controlled according to the previously obtained bound of $\io n^pc^{-r}$ and the second is susceptible to an application of the Gagliardo--Nirenberg inequality and subsequent absorption by the diffusion term (see proof of Lemma \ref{lem;q;c}).  
This will enable us to transform the information on $\io n^pc^{-r}$ into a boundedness assertion on $\norm[\Lom p]{n}$ for some suitably large $p$ (Lemma \ref{lem;Lp;n}). Having derived (time-local) $\Lom{\infty}$-bounds for the fluid velocity field in the Stokes- and Navier--Stokes settings in Sections \ref{sec:stokes2d3d} and \ref{sec:navierstokes2d}, respectively, in Section \ref{sec:boundednessforn} we can, mainly leaning on semigroup estimates, conclude global existence of solutions. 

The second regard in which the fluid poses an obstacle concerns the crucial uniform-in-time lower bound for $c$ constituting the core of the boundedness proof in \cite{fujie}. Again viewing the equation for $c$ as inhomogeneous heat equation with source term $n$, in 
\cite[Lemma 2.2]{fujie} estimates of the heat kernel provide this uniform positive lower bound on $c$. These estimates rely on the nonnegativity of $n$, whereas no information about the sign of $u\cdot \na c$ (and hence of $n-u\cdot\na c$) seems available. That we are hence lacking the corresponding time-global positive lower bound of $c$ 
is the main reason why we have to leave open the question of boundedness of the solutions obtained in Theorem \ref{mainthm}.

\section{Basic properties and estimates}

%
%
%
%
We first recall a local existence result. 
We also give some lower estimate for $c$, 
which plays an important role to in 
avoiding the difficulty of the 
singular sensitivity function. 

\begin{lem}\label{lem;local existence}
Let $N\in\{2,3\}$, $\chi>0$, $\kappa\in\{0,1\}$, $\vartheta>N$, $\alpha\in(\frac{N}{4},1)$ and let $\Omega\subset \RN$ be a 
bounded domain with smooth boundary. 
Assume that $n_0,c_0,u_0,\phi$ satisfy 
\eqref{condi;ini1}{\rm --}\eqref{condi;ini2}. 
Then there exist $\tmax\in (0,\infty]$ and 
a classical solution $(n,c,u,P)$ of \eqref{cp}--\eqref{init}
in $\Omega\times (0,\tmax)$ such that 
\begin{align*}
  &n\in C^0 (\overline{\Omega}\times [0,\tmax))\cap 
  C^{2,1}(\overline{\Omega}\times (0,\tmax)), 
  \\ 
  &c\in C^0 (\overline{\Omega}\times [0,\tmax))\cap 
  C^{2,1}(\overline{\Omega}\times (0,\tmax))
  \cap L^\infty_{\rm loc}([0,\tmax);W^{1,\vartheta}(\Omega)),
  \\ 
  &u\in C^0 (\overline{\Omega}\times [0,\tmax))\cap 
  C^{2,1}(\overline{\Omega}\times (0,\tmax)),\\
  &P \in C^{1, 0}(\overline{\Omega}\times(0, \tmax))  
\end{align*}
and 
\begin{align*}
  \tmax=\infty \quad \mbox{or}\quad 
  \lim_{t\to \tmax}
  \left(
  \lp{\infty}{n(\cdot,t)}
  + \|c(\cdot,t)\|_{W^{1,\vartheta}(\Omega)}
  +\lp{2}{A^\alpha u(\cdot,t)}
  \right)=\infty. 
\end{align*}
Also, the solution is unique, 
up to addition of spatially constant function to $P$ and, moreover, has the properties 
\begin{align}\notag
  &n(x,t)\ge 0 
  \quad \mbox{and}
  \\\label{ineq;lower;c} 
  &c(x,t)\ge 
  \left(\min_{x\in\ol{\Omega}}c_0(x)\right)e^{-t}
 \quad \mbox{for all}\ t\in (0,\tmax). 
\end{align}
\end{lem}
\begin{proof}
With adaptions akin to those used in \cite[Thm. 2.3 i)]{lankeit_m2as} to deal with the singular sensitivity, the usual reasoning (see \cite[Thm. 3.1]{BBTW} and \cite[L. 2.1]{win_CTNS_global_largedata}) based on Banach's fixed point theorem applied in a closed bounded set in $L^\infty((0,T);C^0(\Ombar)\times W^{1,\vartheta}(\Om)\times D(A^\alpha))$ for suitably small $T>0$, followed by regularity arguments, proves this local existence and uniqueness result. The estimates in \eqref{ineq;lower;c} immediately follow from the comparison principle.
\end{proof}

In the following, we will always assume $N$, $\Om$, $\kappa$, $\chi$, $n_0$, $c_0$, $u_0$, $\phi$, $\vartheta$, $\alpha$ to satisfy the conditions of Lemma \ref{lem;local existence} and to be fixed. By $(n,c,u,P)$ we will denote the corresponding solution to \eqref{cp}--\eqref{init} given by Lemma \ref{lem;local existence} and by $\tmax$ its maximal existence time. 

%
%
%
%

Our study of these solutions begins with the following simple $\Lom1$-information: 

\begin{lem}\label{lem;L1esti;nc}
For all $t\in(0,\tmax)$ the mass identity 
  \begin{align}\label{id;mass}
  \into n(\cdot,t) = \into n_0
  \end{align}
is satisfied.   Moreover, there exists $C>0$ such that 
  \begin{align*} 
  \lp{1}{c(\cdot,t)}\leq C 
  \quad 
  \mbox{for all}\ t\in(0,\tmax). 
  \end{align*} 
\end{lem}
\begin{proof}
This results from integration of the first and second equation of \eqref{cp} due to $\na\cdot u=0$ in $\Om\times(0,\tmax)$ and \eqref{boundary}.
\end{proof}

%
%
%
%
As in related situations (see \cite{win_singular,fujie}, but also \cite{lankeit_winkler}), the key for establishing estimates significantly going beyond those of Lemma \ref{lem;L1esti;nc} lies in the following: 
\begin{lem}\label{lem;intnpc-r}
 If $\chi<1$, $p\in (1,\frac{1}{\chi^2})$ and $r\in I_p$, 
 where 
 \begin{align*}
   I_p:=\left(
        \frac{p-1}{2}\left(1-\sqrt{1-p\chi^2}\right),
        \frac{p-1}{2}\left(1+\sqrt{1-p\chi^2}\right)
        \right),
 \end{align*}
 then there are $C_1>0$, $C_2>0$ such that 
 \begin{align}\label{eq:ionpcrbd}
   \into n(\cdot,t)^p c(\cdot,t)^{-r} 
   \le C_1 e^{C_2t} 
   \quad 
   \mbox{for all}\ t\in(0,\tmax) 
 \end{align}
 and, moreover, for any finite $T\in(0,\tmax]$ there exists $C(T)>0$ such that 
 \begin{align}\label{timeint;nablanc}
 \int_0^{T}\into 
 |\nabla (n^{\frac{p}{2}}c^{-\frac{r}{2}})|^2 
 \le C(T). 
\end{align}
\end{lem}
\begin{proof} 
Using \eqref{cp} and integration by parts, we have that 
\begin{align}\notag
  \frac d{dt} \into n^p c^{-r} 
  &= -p(p-1)\into n^{p-2}c^{-r} |\nabla n|^2 
     + (2pr +\chi p(p-1))
        \into n^{p-1}c^{-r-1}\nabla n\cdot \nabla c 
\\\notag
  &\quad\, -\left(\chi pr +r(r+1)\right)
        \into n^p c^{-r-2}|\nabla c|^2 
        + r\into n^pc^{-r} - r\into n^{p+1}c^{-r-1}  
\\\label{eq;ddtintupcr}
  &\quad\, -\into u\cdot \nabla \left(n^pc^{-r}\right) \quad \text{on } (0,\tmax).
\end{align}
The condition $r\in I_p$ entails that 
\[
 0 > 4p\left(r^2-(p-1)r+\f{p(p-1)^2\chi^2}4\right) = \kl{2pr+\chi p(p-1)}^2- 4p(p-1)(\chi p r +r (r+1))
\]
and hence 
\begin{equation*}
 (2pr+\chi p (p-1))^2 < 4p(p-1)(\chi pr+r(r+1)). 
\end{equation*}
Therefore, we can apply Young's inequality to the summand $(2pr+\chi p(p-1))\io n^{p-1} c^{-r} \na n\cdot \na c$ in \eqref{eq;ddtintupcr} and with some small $\eps>0$ obtain 
\begin{equation}\label{eq:ddtionpcmr}
 \f{d}{dt} \io  n^pc^{-r} + \eps \io n^{p-2}c^{-r} |\na n|^2 + \eps \io n^p c^{-r-2} |\na c|^2 \le r \io n^p c^{-r} \quad \text{in } (0,\tmax), 
\end{equation}
where we already have taken into account that $\io u\cdot \na(n^p c^{-r})=0$ and $-r\io n^{p+1}c^{-r-1}\le 0$ in $(0,\tmax)$. Integration of \eqref{eq:ddtionpcmr} shows \eqref{eq:ionpcrbd} and, since 
\begin{align*}
  \frac{p^2}{2}n^{p-2}c^{-r} |\nabla n|^2 
  + \frac{r^2}{2}n^p c^{-r-2}|\nabla c|^2
  \geq  |\nabla (n^{\frac p2}c^{-\frac r2})|^2 \qquad \text{in } \Om\times(0,\tmax), 
\end{align*}
also results in \eqref{timeint;nablanc}.

\end{proof}

%
%
%
%
In order to extract helpful boundedness information concerning $n$ from \eqref{eq:ionpcrbd}, we first require estimates for higher norms of $c$, whose source is the following lemma: 

%
\begin{lem}\label{lem;id;DIforcq}
For all $q\geq 1$, 
\begin{align}\label{id;DIforcq}
  \frac 1q \frac{d}{dt}\into c^q 
  = -(q-1)\into c^{q-2}|\nabla c|^2 
    -\into c^q  
    + \into nc^{q-1}
\end{align}
holds on $(0,\tmax)$. 
\end{lem}
\begin{proof}
From the second equation in \eqref{cp} 
we obtain that 
\begin{align}\label{ineq;dif;cq}
  \frac 1q\frac{d}{dt}\into c^q 
  = \into c^{q-1}\Delta c - \into c^q 
  +\into nc^{q-1} - \into c^{q-1} u\cdot \nabla c
\end{align}
on $(0,\tmax)$. 
Here we note from $\nabla\cdot u=0$ 
in $\Omega\times (0,\tmax)$ that  
\begin{align*}
  \into c^{q-1}u\cdot \nabla c
  =\frac 1q\into u\cdot \nabla \left(c^q\right)
  =0 
\end{align*}
on $(0,\tmax)$ and therefore the last term in  
\eqref{ineq;dif;cq} vanishes, so that 
\eqref{id;DIforcq} holds. 
\end{proof}
%
%
%
%
We now turn our attention to the derivation of $L^q(\Omega)$-estimates for $c$.  
In light of the differential inequality from 
Lemma \ref{lem;id;DIforcq}, our main objective will be 
the estimate of $\into nc^{q-1}$.

\begin{lem}\label{lem;q;c}
If $\chi<\sqrt{\f2N}$, for all $q\in(1,\infty)$ and any finite $T\in(0,\tmax]$ 
there exists a constant $C(q,T)>0$ such that 
\begin{align*}
  \lp{q}{c(\cdot,t)}\le C(q,T)
  \quad \mbox{for all}\ t\in (0,T). 
\end{align*}
\end{lem}
\begin{proof} 
Without loss of generality we may assume $q\ge 2$. 
We put $p\in (\frac{N}{2},\min\{\frac{1}{\chi^2},3\})$ and 
$r=\frac{p-1}{2}$, so that the inequality $q>p-r$ holds and, clearly, $r\in I_p$. Therefore Lemma \ref{lem;intnpc-r} is applicable and asserts the existence of $C_1>0$ and $C_2>0$ such that, by Hölder's inequality, 
%
%
\begin{align}\label{esti;intncq-1}
  \into nc^{q-1} 
  \le \left(\io n^p c^{-r}\right)^{\frac 1p}
      \left(\int c^{\frac{pq-p+r}{p-1}}\right)^{\frac{p-1}{p}}
  \le C_1e^{C_2T}
  \lp{\frac{2(pq-p+r)}{q(p-1)}}
  {c(\cdot,t)^{\frac{q}{2}}}^{a}
\end{align}
in $(0,T)$, 
where $a:=\frac{2(pq-p+r)}{pq}$. 
Due to the Gagliardo--Nirenberg inequality there is
$C_{GN}>0$ such that 
\begin{align}\label{G-N}
  &\lp{\frac{2(pq-p+r)}{q(p-1)}}{c(\cdot,t)^{\frac{q}{2}}}^a
  &\le 
  C_{GN}
  \lp{2}{\nabla c(\cdot,t)^{\frac{q}{2}}}^{ab}
  \lp{2}{c(\cdot,t)^{\frac{q}{2}}}^{a(1-b)}
  + C_{GN} \lp{\frac{2}{q}}{c(\cdot,t)^{\frac{q}{2}}}^a 
\end{align}
for all $t\in (0,T)$, 
where $b:=\frac{N(q-p+r)}{2(pq-p+r)}$. 
Noting from $-N(p-r) < (2p-N)q$ that 
\begin{align*}
  ab=
  {\frac{2(pq-p+r)}{pq}}\cdot 
  \frac{N(q-p+r)}{2(pq-p+r)} 
  =\frac{N(q-p+r)}{pq}
  <2,
\end{align*}

we infer from the Young inequality that with some $C_3(q,T)>0$
\begin{align}
\lp{2}{\nabla c(\cdot,t)^{\frac{q}{2}}}
  ^{ab}
  \lp{2}{c(\cdot,t)^{\frac{q}{2}}}
  ^{a(1-b)}
\label{cq;ineq;first}
  &\le 
  \frac{q-1}{2C_{GN}C_1e^{C_2T}} 
  \lp{2}{\nabla c(\cdot,t)^{\frac{q}{2}}}^{2}
  +
  C_3(q,T)\lp{2}{c(\cdot,t)^{\frac{q}{2}}}
  ^{\frac{2a(1-b)pq}{2pq-N(q-p+r)}}
\end{align}
for all $t\in (0,T)$.
Here since $N\in\{2,3\}$ implies $(N-1)r<p$, 
we can confirm that 
\begin{align*}
  \frac{2a(1-b)pq}{2pq-N(q-p+r)}=
  \frac{2(2pq+(N-2)p-Nq+(N-2)r)}{2pq-N(q-p+r)}<2. 
\end{align*}
Thus, relying once more on the Young inequality 
we can derive from \eqref{cq;ineq;first}
that there exists $C_4(q,T)>0$ 
such that 
\begin{align}\notag
&\lp{2}{\nabla c(\cdot,t)^{\frac{q}{2}}}^{ab}
  \lp{2}{c(\cdot,t)^{\frac{q}{2}}}^{a(1-b)}
  \\\label{cq;ineq;second}
  &\le 
  \frac{q-1}{2^{a+1}C_{GN}C_1e^{C_2T}} 
  \lp{2}{\nabla c(\cdot,t)^{\frac{q}{2}}}
  ^{2}
  +
  \frac{1}{2^{a+1}C_{GN}C_1e^{C_2T}}
  \lp{2}{c(\cdot,t)^{\frac{q}{2}}}^{2}
  + C_4(q,T)
\end{align}
is valid for all $t\in (0,T)$. 
Combination of \eqref{id;DIforcq}, \eqref{esti;intncq-1}, \eqref{G-N} and \eqref{cq;ineq;second}  
with Lemma \ref{lem;L1esti;nc} implies that with some $C_5(q,T)>0$ 
\begin{align*}
  \frac 1q\frac{d}{dt} \into c^q 
  \le -\frac{q-1}{2}\into c^{q-2}|\nabla c|^2 
    -\frac{1}{2}\into c^q
    + C_5(q,T) 
\end{align*}
on $(0,T)$, 
which means that there exists $C_6(q,T)>0$ 
such that 
%
\[
\lp{q}{c(\cdot,t)} \le C_6(q,T) 
\quad \mbox{for all}\ t\in (0,T).\qedhere
\]
%
\end{proof}
%

%
%
%
%
Now we have already prepared all tools to obtain an 
$\Lom p$-estimate for $n$, for some $p>\f N2$, which is an important stopover on the route to the $\Lom\infty$-estimate for $n$ and will be of particular importance in the proofs of Lemma \ref{lem;esti;nablav} and Lemma \ref{lem;esti;stokes}. 
\begin{lem}\label{lem;Lp;n}
 We assume $\chi<\sqrt{\f2N}$. 
  For any $p \in [1,\frac{1}{\chi^2})$ 
  and any finite $T\in(0,\tmax]$ 
  there is $C(p,T)>0$ such that 
  \begin{align}\label{nbd}
   \lp{p}{n(\cdot, t)} 
   \le C(p,T)
   \quad \mbox{for all}\ t\in (0,T). 
  \end{align}
\end{lem}
\begin{proof}
Without loss of generality we assume $p\in (\frac{N}{2},\frac{1}{\chi^2})$, let $p_0\in (p,\frac{1}{\chi^2})$ and set 
$r_0:= \frac{p_0-1}{2}\in I_p$.  
Then we can see from H\"older's inequality that 
\begin{align*}
 \io n^p = \io \kl{n^{p_0}c^{-r_0}}^{\f{r_0p}{p_0}} \le \kl{\io n^{p_0}c^{-r_0}}^{\f p{p_0}} \kl{\io c^{\f{r_0 p}{p_0-p}}}^{\f{p_0-p}{p_0}} 
\end{align*}
on $(0,T)$, which by Lemmata \ref{lem;intnpc-r} and \ref{lem;q;c} implies \eqref{nbd}.
\end{proof}

\section{Boundedness for $u$}
Having obtained $\Lom p$-bounds for $n$ and hence for the driving force in the fluid equation, we devote this section to the derivation of estimates for the fluid velocity. We begin with the following $\Lom2$-information on $u$ and $L^2(\Om\times(0,T))$-estimate for $\na u$, before we separately consider the cases of Stokes- and Navier--Stokes--fluids in Subsections \ref{sec:stokes2d3d} and \ref{sec:navierstokes2d}.
%
%
%
%

\begin{lem}\label{lem;L2estiforu}
If $\chi<\sqrt{\f2N}$, for any finite $T\in(0,\tmax]$ there exists $C(T)>0$ such that 
 \begin{align}\label{eq:ul2}
    \into |u(\cdot,t)|^2\le C(T) 
    \quad \mbox{for all}\ t\in(0,T)
 \end{align}
and 
 \begin{align}\label{eq:naul2}
    \int_0^{T}\into |\nabla u(\cdot,t)|^2\le C(T). 
 \end{align}
\end{lem} 
\begin{proof}
 Testing the third equation in \eqref{cp} by $u$ and 
 integrating by parts, we see that 
 \begin{align}\label{equ;L2;u}
 \frac 12 \frac d{dt}\into |u|^2 
 + \into |\nabla u|^2 
 = \into nu\cdot \nabla \phi \quad \text{on }(0,T), 
 \end{align}
 because $\kappa \io u\cdot (u\cdot\na)u = 
 -\kappa\io (\na\cdot u)|u|^2=0$. 
 We let $p\in(\f{N+2}{2N},\f1{\chi^2})$ so that Lemma \ref{lem;Lp;n} provides us with $C_1>0$ such that $\norm[\Lom p]{n(\cdot,t)}\le C_1$ for all $t\in(0,T)$. Then setting $p':=\f{p}{p-1}$ we have $p'\in[1,\f{2N}{N-2})$ and may rely on the Sobolev embedding $W^{1,2}(\Om)\hookrightarrow \Lom {p'}$ to obtain $C_2>0$ satisfying $\norm[\Lom{p'}]{w}\le C_2 \norm[\Lom2]{\na w}$ for all $w\in W_0^{1,2}(\Om)$. From Hölder's inequality, this embedding, and Young's inequality, we can conclude that 
\begin{align*}
  \into nu\cdot \nabla \phi 
  \le C_2\lp{\infty}{\nabla \phi} 
  \lp{p}{n}\lp{2}{\nabla u}
  \le \frac 12\into |\nabla u|^2 + 
  \frac{C_1^2C_2^2\lp{\infty}{\nabla \phi}^2 
  }{2} 
\end{align*} 
holds on $(0,T)$, 
which by \eqref{equ;L2;u} implies that 
there exists $C_3>0$ such that 
\begin{align}\label{ineq;L2;u}
 \frac d{dt}\into |u|^2 
 + \into |\nabla u|^2 
 \le C_3 \quad \text{on } (0,T).
 \end{align}
 Thus thanks to the Poincar\'e inequality 
 we can find a constant such that \eqref{eq:ul2} holds. The combination of \eqref{ineq;L2;u} and \eqref{eq:ul2} then also entails  \eqref{eq:naul2} with some $C(T)$.
\end{proof}

\subsection{The case $\kappa =0$}\label{sec:stokes2d3d}
In the case $\kappa=0$ the regularity properties for $n$ and $u$ already established in Lemma \ref{lem;Lp;n} and Lemma \ref{lem;L2estiforu}, respectively, will be sufficient to prove the boundedness of $u$ even in the case of $N=3$. It is well known  that the regularity of solutions to the Stokes subsystem $u_t+Au=\HP\left(n\nabla\phi\right)$ appearing in \eqref{cp} is only contingent on the regularity of the forcing term $\HP\left(n\nabla\phi\right)$. Arguments appearing in the proof of the lemma below have been previously used in e.g. \cite{Winkler_2015_general_sensitivity} and \cite{wang_xiang_15_2d} and rely on semigroup estimates for the Stokes semigroup.

\begin{lem}\label{lem;esti;stokes}
 If $\chi<\sqrt{\f2N}$, for any finite $T\in(0,\tmax]$ and any $\alpha_0\in(\f N4,\alpha]$ satisfying 
 \[
  \alpha_0<1-\f N2 \chi^2 +\f N4,
 \]
 there is $C(T)>0$ such that 
 \begin{equation}\label{eq;esti;stokes}
  \norm[\Lom2]{A^{\alpha_0}\kl{e^{-tA}u_0 + \int_0^t e^{-(t-s)A} \calP(n(\cdot,s)\na\phi)ds}}\le C(T) \qquad \text{for all } t\in(0,T).
 \end{equation}
\end{lem}
\begin{proof}
 We pick $p\in (\f N2,\min\{\f1{\chi^2},2\})$ and $\delta\in(0,1)$ sufficiently small such that 
\begin{align*}
\alpha_0+\delta<1-\frac{N}{2}\left(\frac{1}{p}-\frac{1}{2}\right)
\end{align*}
 holds. We then fix $p_0>p$ satisfying 
 \begin{align*}
2\delta-\frac{N}{p}>-\frac{N}{p_0}
\end{align*}
and note that 
\begin{align}\label{eq:lem-u-eq1}
\alpha_0+\delta+\frac{N}{2}\left(\frac{1}{p_0}-\frac{1}{2}\right)<1. 
\end{align}
Since $u_0\in D(A^{\alpha_0})$ by $\alpha_0\le \alpha$, we can use that $A^{\alpha_0}$ and $e^{-tA}$ commute (\cite[p.\ 206, (1.5.16)]{sohr_book}) and thereby find $C_1>0$ such that 
\begin{equation}\label{eq:lem-u-eq2}
 \norm[\Lom2]{A^{\alpha_0}e^{-tA}u_0}=\norm[\Lom2]{e^{-tA}A^{\alpha_0}u_0} \le \norm[\Lom2]{A^{\alpha_0}u_0}\le C_1 \qquad \text{for all } t\in(0,\tmax). 
\end{equation}
To treat the integrand in \eqref{eq;esti;stokes}, we recall that by standard regularity estimates for the Stokes semigroup (e.g. \cite[Lemma 3.1]{Winkler_2015_general_sensitivity}, 
\cite[Lemma 2.3]{caolan16}) there exist $C_2>0$ and $\lambda_1>0$ such that
\begin{align}\label{eq:lem-u-eq4}
&\big\|A^{\alpha_0+\delta}e^{-(t-s)A}A^{-\delta}\HP\left((n(\cdot,s)\nabla\phi\right)\big\|_{L^{2}\left(\Omega\right)}\nonumber\\
&\leq\  
C_2(t-s)^{-\alpha_0-\delta-\frac{N}{2}(\frac{1}{p_0}-\frac{1}{2})}e^{-\lambda_1(t-s)}\big\|A^{-\delta}\HP\left(n(\cdot,s)\nabla\phi\right)\big\|_{L^{p_0}\left(\Omega\right)}
\end{align}
for all $s\in(0,t)$. Additionally, our choice of $2\delta-\frac{N}{p}>-\frac{N}{p_0}$ implies that for any $s\in(0,t)$ we have
\begin{align}\label{eq:lem-u-eq5}
\big\|A^{-\delta}\HP\left(n(\cdot,s)\nabla\phi\right)\big\|_{L^{p_0}\left(\Omega\right)}\leq C_3\|n(\cdot,s)\nabla\phi\|_{L^p\left(\Omega\right)}\leq C_3\|\nabla\phi\|_{L^\infty\left(\Omega\right)}\|n(\cdot,s)\|_{L^p\left(\Omega\right)},
\end{align}
with some $C_3>0$ (see also \cite[Lemma 3.3]{Winkler_2015_general_sensitivity} and \cite[Lemma 2.3]{wang_xiang_15_2d}). Due to the fact that $p\in(\frac{N}{2},\frac{1}{\chi^2})$, Lemma \ref{lem;Lp;n} provides $C_4>0$ such that $\|n(\cdot,t)\|_{L^p\left(\Omega\right)}\leq C_4$ for all $t\in(0,T)$ and therefore a combination of \eqref{eq:lem-u-eq2}--\eqref{eq:lem-u-eq5} shows that
\begin{align*}
&\norm[\Lom2]{A^{\alpha_0}\kl{e^{-tA}u_0 + \int_0^t e^{-(t-s)A} \calP(n(\cdot,s)\na\phi)ds}}\\ 
&\leq C_1+C_2C_3C_4\|\nabla \phi\|_{L^\infty\left(\Omega\right)}\int_0^t(t-s)^{-\alpha_0-\delta-\frac{N}{2}(\frac{1}{p_0}-\frac{1}{2})}e^{-\lambda_1(t-s)} d s
\end{align*}
holds for all $t\in(0,T)$, which in view of \eqref{eq:lem-u-eq1} and the fact that $\phi\in C^2\left(\overline{\Omega}\right)$ implies \eqref{eq;esti;stokes}. 
\end{proof}

\begin{lem}\label{lem;esti;u}
Assume $\chi<\sqrt{\f 2N}$ and let $\alpha_0\in(\f N4,\alpha]$ satisfy $\alpha_0<1-\f N2 \chi^2 +\f N4$. 
If $\kappa=0$, corresponding to any finite $T\in(0,\tmax]$ there exists $C(T)>0$ such that
\begin{equation}\label{eq:Aalpha-u-stokes-bd}
 \norm[\Lom2]{A^{\alpha_0}u(\cdot,t)}\le C(T) \quad \text{for all } t\in(0,T)
\end{equation}
and 
\begin{align}\label{eq:u-stokes-bd}
\|u(\cdot,t)\|_{L^\infty\left(\Omega\right)}\leq C(T)\quad\text{for all }t\in(0,T).
\end{align}
\end{lem}
\begin{proof}
 Since for $\kappa=0$ the solution $u$ is given by 
 \[
  u(\cdot,t)=e^{-tA} u_0 + \int_0^t e^{-(t-s)A} \calP \kl{n(\cdot,s)\na \phi} ds, 
 \]
 due to the compatibility of the choice of $\alpha_0$ with the requirements of Lemma \ref{lem;esti;stokes}, this lemma immediately yields \eqref{eq:Aalpha-u-stokes-bd}, whereupon \eqref{eq:u-stokes-bd} is a consequence of the embedding $D(A^{\alpha_0})\hookrightarrow C^{\gamma}(\Om)$ for arbitrary $\gamma\in(0,2\alpha_0-\f N2)$. 
\end{proof}

\subsection{The case $\kappa =1$ and $N=2$}\label{sec:navierstokes2d}
Here we will focus on the case that $\kappa=1$ and $N=2$. 
In this setting we can make use of arguments 
previously employed in \cite{win_CTNS_global_largedata}. 

%
%
%
%

The proof of Lemma \ref{lem:nauAu} will require some spatio-temporal integrability of $n$ of higher order than directly guaranteed by Lemma \ref{lem;Lp;n}. We therefore prepare the following:  

\begin{lem}\label{lem;timeintnpcr} 
  Assume that $\kappa=1$ and $N=2$. 
  If $\chi<\sqrt{\f 2N}$, for any finite $T\in(0,\tmax]$ 
  there is $C(T)>0$ 
  such that 
  \begin{align*}
    \int_0^{T} \into n^2 \le C(T). 
  \end{align*}
\end{lem}
\begin{proof}
Let $p\in (1,\frac{1}{\chi^2})$ and 
$r:=\frac{p-1}{2}\in I_p$. 
First we recall from \eqref{timeint;nablanc} 
that there are $C_1>0$ and $C_2>0$ such that 
 \begin{align*}
 \int_0^{T}\into 
 |\nabla (n^{\frac{p}{2}}c^{-\frac{r}{2}})|^2 
 \le C_1 \quad\text{and }\quad \sup_{t\in(0,T)} \io n^pc^{-r}  (\cdot,t)\le C_2 
\end{align*}
hold.  
The Gagliardo--Nirenberg inequality hence provides $C_3>0$ such that 
\begin{align*}
  \int_0^{T} 
  \lp{2p}{n(\cdot,t)c(\cdot,t)^{-\frac{r}{p}}}^{2p}\,dt 
  &= \int_0^{T} 
  \lp{4}{n(\cdot,t)^{\frac{p}{2}}
  c(\cdot,t)^{-\frac{r}{2}}}^{4}\,dt 
  \\
  &\le C_3 
  \int_0^{T} 
  \lp{2}{\nabla (n(\cdot,t)^{\frac{p}{2}}c(\cdot,t)^{-\frac{r}{2}})}^{2}
  \lp{2}{n(\cdot,t)^{\frac{p}{2}}c(\cdot,t)^{-\frac{r}{2}}}^{2}\,dt 
  \\
  &\quad\, 
  + C_3 \int_0^{T} 
  \lp{2}{n(\cdot,t)^{\frac{p}{2}}c(\cdot,t)^{-\frac{r}{2}}}\,dt  \le C_3C_2C_1 + C_3C_2T,
\end{align*}
which means that 
  \begin{align}\label{timeintnpcr}
    \int_0^{T}\into n^{2p}c^{-(p-1)}\le C
  \end{align}
holds. Thanks to Young's inequality, we can estimate 
\begin{align*}
  \int_0^{T} \into n^2 
  \le \int_0^{T} \into n^{2p} c^{-(p-1)} 
  + \int_0^{T} \into c, 
\end{align*} 
so that \eqref{timeintnpcr} and the $\Lom{1}$-boundedness of $c$ as asserted by 
Lemma \ref{lem;L1esti;nc} finish the proof.
\end{proof}
%
%
%
%
%
\begin{lem}\label{lem:nauAu}
 If $\chi<\sqrt{\f 2N}$, $\kappa=1$, $N=2$, then for any finite $T\in(0,\tmax]$ there exists $C(T)>0$ such that 
  \begin{align*}
    \into |\nabla u(\cdot,t)|^2 \le C(T) 
    \quad \mbox{for all}\ t\in (0,T)
  \end{align*}
  and 
  \begin{align}\label{Aubd}
    \int_0^{T} \into |Au|^2 \le C(T). 
  \end{align}
\end{lem}
\begin{proof}
Testing the third equation in \eqref{cp} by $Au$ and 
using the Young inequality, we obtain that 
\begin{align}\label{diffineq:nau2}
  \frac{1}{2}\frac{d}{dt}\into |\nabla u|^2 
  + \into |Au|^2 
  &= \into (n\nabla \phi) Au 
  - \into (u\cdot \nabla)u Au \nn
  \\
  &\le \frac{1}{2}\into |Au|^2 
  + C_1 \into n^2 + \into |u|^2 |\nabla u|^2 \quad \text{on } (0,T)
\end{align}
with some $C_1>0$. 
Concerning the last term herein, we follow 
\cite[Proof (of Theorem 1.1)]{win_CTNS_global_largedata} and employ the Gagliardo--Nirenberg inequality, Lemma \ref{lem;L2estiforu} and Young's inequality in obtaining $C_2>0$, $C_3>0$ and $C_4>0$ such that  
\begin{align*}
 \io |u|^2|\na u|^2 &\le \norm[\Lom{\infty}]{u}^2 \norm[\Lom2]{\na u}^2 \\
 &\le C_2 \norm[W^{2,2}(\Om)]{u} \norm[\Lom2]{u}  \norm[\Lom2]{\na u}^2\\
 &\le C_2 C_3 \norm[W^{2,2}(\Om)]{u}  \norm[\Lom2]{\na u}^2\\
 &\le \frac14 \norm[\Lom2]{Au}^2 + C_4 \norm[\Lom2]{\na u}^4 \qquad \text{on } (0,T).
\end{align*}
From \eqref{diffineq:nau2} we can hence see that with $C_5:=\max\{C_1,C_4\}>0$ 
\begin{align*}
  \frac d{dt} \into |\nabla u|^2 
  + \frac 14\into |Au|^2 
  \le 
  C_5\left(\Big(\into |\nabla u|^2\Big)^2
  + \into n^2\right) \qquad \text{on } (0,T). 
\end{align*}
If we put $y(t):=\into |\nabla u(\cdot,t)|^2$, $t\in (0,T)$, then $y$ satisfies 
\begin{align*}
  y'(t)\le C_5\left(\Big(\io |\na u(\cdot,t)|^2\Big) y(t) + \into n^2(\cdot,t)\right) 
  \quad \text{for all } t\in (0,T), 
\end{align*}
which implies 
\begin{align*}
  y(t) \le 
  y(0)e^{C_5\int_0^t \io |\na u(\cdot,s)|^2 \,ds}
  + C_5 \int_0^t e^{C_5\int_s^t \io |\na u(\cdot,\sigma)|^2\,d\sigma} 
  \Big(\into n^2(x,s)\,dx\Big)\,ds, \quad t\in(0,T).  
\end{align*}
Here noting from Lemmata \ref{lem;L2estiforu} 
and \ref{lem;timeintnpcr} that both $\int_0^T\io |\na u(\cdot,\sigma)|^2d\sigma$ and $\int_0^T\io n^2(\cdot,s)\,ds$ are finite,  
we find $C_6(T)>0$ such that 
\begin{align*}
  \into |\nabla u|^2 \le C_6(T)\qquad \text{on } (0,T). 
\end{align*}
Thanks to this boundedness, we can establish \eqref{Aubd}.
%
\end{proof}

%
%
%
%
%

Then the same argument as in \cite[Proof (of Theorem 1.1)]{win_CTNS_global_largedata} 
leads to the $\Lom{\infty}$-estimate for $u$:  

\begin{lem}\label{lem;esti;u;Navier}
  If $\chi<\sqrt{\f 2N}$, $\kappa=1$, $N=2$, then for any finite $T\in(0,\tmax]$ and any $\alpha_0\in(\frac{N}{4},\alpha]$ there exists $C_1(T)>0$ such that 
  \begin{align}\label{Aalpha;navier}
  \lp{2}{A^{\alpha_0} u(\cdot,t)}\le C_1(T)
  \quad \mbox{for all}\ t\in (0,T).
  \end{align}
  Moreover, there exists $C_2(T)>0$ satisfying 
  \begin{align}\label{u;infty;navier}
  \lp{\infty}{u(\cdot,t)}\le C_2(T)
  \quad \mbox{for all}\ t\in (0,T).
  \end{align}
\end{lem}

\begin{proof}
 Since $\alpha_0$ satisfies the condition of Lemma \ref{lem;esti;stokes}, we see that 
 \begin{align*}
  \norm[\Lom2]{A^{\alpha_0 } u(\cdot,t)} &\le \norm[\Lom2]{A^{\alpha_0 }\kl{e^{-tA}u_0+\int_0^t e^{-(t-s)A} \calP\big( n(\cdot,s)\na \phi\big) \,ds}} \\&\qquad + \int_0^t \norm[\Lom2]{A^{\alpha_0 } e^{-(t-s)A}\calP\big( (u\cdot \na)u\big)(\cdot,s)} ds \\
  &\le C_{3} + \int_0^t \norm[\Lom2]{A^{\alpha_0 } e^{-(t-s)A}\calP \big((u\cdot \na)u\big)(\cdot,s)} ds\quad \text{for all } t\in(0,T),
 \end{align*}
 with $C_{3 }>0$ given by said lemma. 
 Again following \cite[Proof (of Theorem 1.1)]{win_CTNS_global_largedata}, we choose $p>2$ so large that $p':=\f{p}{p-1}$ satisfies $p'\alpha_0 <1$. 
 Then by a well-known estimate for the norm of $A^{\alpha_0 } e^{-tA}$ (see e.g. \cite[Lemma 2.3i)]{caolan16}) and Hölder's inequality, we find that with some $C_{4 }>0$
 \begin{align*}
  \int_0^t \norm[\Lom2]{A^{\alpha_0} e^{-(t-s)A}\calP \big((u\cdot\na)u\big)(\cdot,s)} ds &\le C_{4 } \int_0^t (t-s)^{-\alpha_0} \norm[\Lom2]{(u(\cdot,s)\cdot\na)u(\cdot,s)}ds\\
  &\hspace{-2.5cm}\le C_{4 }y \kl{\int_0^t (t-s)^{-p'\alpha_0}\,ds}^{\f1{p'}} \kl{\int_0^t \norm[\Lom2]{(u(\cdot,s)\cdot\na)u(\cdot,s)}^{p}ds}^{\f1{p}}
 \end{align*}
 for all $t\in(0,T)$. Since $W^{1,2}(\Om)\hookrightarrow L^{p}(\Om)$ due to $N=2$, from this embedding and the Gagliardo--Nirenberg inequality we obtain $C_{5 }>0$ and $C_{6  }>0$, respectively, such that from Hölder's inequality we can infer 
 \begin{align*}
  \int_0^t \norm[\Lom2]{(u(\cdot,s)\cdot\na)u(\cdot,s)}^p ds &\le 
   \int_0^T\norm[\Lom {p}]{u(\cdot,s)}^{p} \norm[\Lom{\f{2p}{p-2}}]{\na u(\cdot,s)}^{p} ds\\
   &\le C_{5 } \int_0^T \norm[\Lom2]{\na u(\cdot,s)}^{p}  \norm[\Lom{\f{2p}{p-2}}]{\na u(\cdot,s)}^p ds\\
   &\le C_{5 }C_{6  } \int_0^T \norm[\Lom2]{\na u(\cdot,s)}^{p}  \norm[\Lom2]{\Delta u(\cdot,s)}^2 \norm[\Lom2]{\na u(\cdot,s)}^{p-2} ds\\
   &\le C_{5 }C_{6  } \bigg(\sup_{t\in(0,T)} \norm[\Lom2]{\na u(\cdot,t)}^{2p-2}\bigg) \int_0^T \norm[\Lom2]{\Delta u(\cdot,s)}^2 ds. 
 \end{align*}
 Here an application of Lemma \ref{lem:nauAu} finishes the proof of \eqref{Aalpha;navier}, which by the embedding of $D(A^{\alpha_0})$ into $\Lom{\infty}$ entails \eqref{u;infty;navier}, too. 
\end{proof}

\section{Boundedness for $n$}\label{sec:boundednessforn}
%
%
%
%
The goal of this section will be to establish an $\Lom\infty$-estimate for $n$ by combination of previously obtained estimates and to finally prove Theorem \ref{mainthm}. Control of the cross-diffusion term in the equation for $n$ will be supplied by the following boundedness statement concerning $\na c$. 

%
\begin{lem}\label{lem;esti;nablav}
If $\kappa=1$, assume that $N=2$. 
Let $1\le p\leq q<\infty$ satisfy  
$q<\vartheta$ and $\frac{1}{2} + 
\frac{N}{2}(\frac{1}{p}-\frac{1}{q})<1$. 
Then for any finite $T\in(0,\tmax]$ there exists a constant $C(T)>0$ such that 
\begin{align}\label{estimate:nac:by:n}
  \lp{q}{\nabla c(\cdot,t)}
  \le 
  C(T) \bigg(
  1+\sup_{s\in (0,T)}\lp{p}{n(\cdot,s)}
  \bigg)\quad \mbox{for all}\ t\in (0,T).
\end{align}
In particular, if $\chi<\sqrt{\f 2N}$, then for any $q\in[1,\f1{\chi^2-\f1N})\cap[1,\vartheta)$ and any finite $T\in(0,\tmax]$ there is $C(q,T)>0$ such that 
\begin{equation}\label{bd:nablac}
 \norm[\Lom q]{\na c(\cdot,t)} \le C(q,T) \qquad \text{for any } t\in(0,T).
\end{equation}

\end{lem}
\begin{proof}
Applying the variation of constants formula for $c$, 
we have 
\begin{align}\label{nablac;semirep}
  c(\cdot,t) 
  &= e^{t(\Delta-1)}c_0 
  + \int_0^t e^{(t-s)(\Delta -1)}
  \big(n(\cdot,s)+u(\cdot,s)\cdot \nabla c(\cdot,s)
  \big)\,ds, \qquad t\in(0,\tmax). 
\end{align}
In light of standard semigroup estimates for the Neumann 
heat semigroup 
(e.g. \cite[Lemma 1.3 (iii)]{win_aggregationvs}) we find $C_1>0$ such that 
\begin{align}\label{nablac;ini}
  \lp{q}{\nabla e^{t(\Delta-1)}c_0}
  \le C_1\lp{\vartheta}{\nabla c_0} \quad \text{for all } t\in(0,\tmax).
\end{align}
Similarly, the semigroup estimates of \cite[Lemma 1.3 (ii)]{win_aggregationvs} provide us with $C_2>0$, for any $t\in(0,T)$ fulfilling 
\begin{align}\notag
  \int_0^t 
  \lp{q}{e^{(t-s)(\Delta -1)}n(\cdot,s)}\,ds
  &\le C_2\int_0^t 
  \left(
  1+(t-s)^{-\frac{1}{2}-\frac{N}{2}(\frac{1}{p}-\frac{1}{q})}
  \right)e^{-(t-s)}\lp{p}{n(\cdot,s)}\,ds 
\\\label{nablac;n}
  &\le C_2\sup_{s\in (0,T)}\lp{p}{n(\cdot,s)}
  \int_0^\infty 
  \left(
  1+\sigma^{-\frac{1}{2}-\frac{N}{2}
  (\frac{1}{p}-\frac{1}{q})}
  \right)e^{-\sigma}\,d\sigma,
\end{align} 
wherein the last integral is finite since $\frac{1}{2}+\frac{N}{2}(\frac{1}{p}-\frac{1}{q})<1$. 
Now we put $r\ge q$, 
$\eta<\frac{1}{2}$ satisfying 
$\frac{1}{2}+\frac{N}{2}
(\frac{1}{r}-\frac{1}{q})<\eta$ 
and $\delta \in (0,\frac{1}{2}-\eta)$. Since $2\eta - \f Nr > 1-\f Nq$, the domain $D((-\Delta+1)^\eta)$ of the fractional power of the operator $-\Delta+1$ is continuously embedded into $W^{1,q}(\Om)$ (see \cite[Theorem 1.6.1]{henry}) and we can hence find $C_3>0$ such that 
\[
 \norm[W^{1,q}(\Om)]{w} \le C_3 \norm[\Lom r]{(-\Delta + 1)^\eta w} \qquad \text{for any } w\in D((-\Delta+1)^\eta). 
\]
Aided by \cite[Lemma 2.1]{Horstmann-Winkler_2005}, 
we moreover fix $C_4>0$ such that 
\[
 \norm[\Lom r]{(-\Delta +1)^\eta e^{-\tau (-\Delta+1)} \na\cdot w}\le C_4\tau^{-\eta-\f12-\delta}e^{-\lambda \tau}\norm[\Lom r]{w} \quad\text{for all } \tau>0 \text{ and } w\in \Lom r.
\]
Additionally relying on $\nabla \cdot u =0$ in 
$\Omega\times (0,\tmax)$ we then find 
%
%
\begin{align}\notag
  \int_0^t
  &\lp{q}{\nabla e^{(t-s)(\Delta-1)} 
  u(\cdot,s)\cdot \nabla c(\cdot,s)}\,ds
\\\notag
  &\le \int_0^t\|e^{(t-s)(\Delta-1)}
  \nabla\cdot (c(\cdot,s)u(\cdot,s))\|_{W^{1,q}(\Omega)}\,ds
\\\notag
  &\le 
  C_3\int_0^t
  \lp{r}
  {(-\Delta +1)^{\eta}e^{(t-s)(\Delta-1)}
  \nabla\cdot (c(\cdot,s)u(\cdot,s))}\,ds
\\\label{nablac;convection1}
  &\le 
  C_3C_4\int_0^t
  (t-s)^{-\eta -\frac{1}{2}-\delta}
  \lp{r}{c(\cdot,s)u(\cdot,s)}\,ds\qquad\text{for all }t\in(0,\tmax).
\end{align}
Since by Lemma \ref{lem;q;c}, 
and Lemma \ref{lem;esti;u} or Lemma \ref{lem;esti;u;Navier} there is $C_5>0$ such that 
\begin{align*} 
  \lp{r}{c(\cdot,s)u(\cdot,s)}
  \le C_5 
  \quad 
  \mbox{for all}\ s\in (0,T),
\end{align*} 
we can conclude from $\eta+\f12+\delta<1$ and \eqref{nablac;convection1} that with $C_6:=C_3C_4C_5\int_0^T \sigma^{-\eta-\f12-\delta} d\sigma \in(0,\infty) $ we have 
\begin{align}\label{nablac;convection2}
  \int_0^t
  &\lp{q}{\nabla e^{(t-s)(\Delta-1)} 
  u(\cdot,s)\cdot \nabla c(\cdot,s)}\,ds
\le C_6 \qquad \text{for all } t\in(0,T).
\end{align}
Combination of 
\eqref{nablac;semirep}, \eqref{nablac;ini}, 
\eqref{nablac;n} and \eqref{nablac;convection2} 
establishes the asserted inequality \eqref{estimate:nac:by:n}.\\
If $q<\f1{\chi^2-\f1N}$ and hence $1>\f12+\f N2(\chi^2-\f1q)$, it is possible to choose $p\in(\f N2,\f1{\chi^2})\cap(\frac{N}{2},q]$ such that still $1>\f12 + \f N2(\f1p-\f1q)$ and \eqref{bd:nablac} results from \eqref{estimate:nac:by:n} and Lemma \ref{lem;Lp;n}.
\end{proof}

%
%
%
%
Now we shall establish a temporally local $\Lom\infty$-estimate for $n$. 
\begin{lem}\label{lem;Linfty;n}
Assume that $\chi<\sqrt{\f2N}$. 
If $\kappa=1$, additionally suppose that $N=2$. 
Then for any finite $T\in(0,\tmax]$ there exists a constant $C(T)>0$ 
satisfying 
\begin{align*}
  \lp{\infty}{n(\cdot,t)} 
  \le C(T)
  \quad \mbox{for all}\ t\in(0,T). 
\end{align*} 
\end{lem}
\begin{proof}
From \eqref{ineq;lower;c} and \eqref{id;mass} we obtain $C_1>0$ and $C_2>0$ such that 
\begin{equation}\label{def:c1c2}
 \inf_{x\in \Om} c(x,t) \ge \f1{C_1} \qquad \text{and} \qquad \io n(\cdot,t)=C_2\qquad \text{for all } t\in(0,T).
\end{equation}
We pick $q,r\in(1,\vartheta)$ such that 
\[
 \f1{\chi^2-\f1N} >q >r >N
\]
and let $C_3>0$ fulfil 
\begin{equation}\label{def:c3}
 \norm[\Lom q]{\na c(\cdot,t)} \le C_3 \qquad \text{for all } t\in (0,T)
\end{equation}
by Lemma \ref{lem;esti;nablav}.
Now for all $T'\in(0,T)$ we note that 
\begin{align*}
M(T'):=\sup_{t\in (0,T')}\lp{\infty}{n(\cdot,t)} 
\end{align*}
is finite and that, furthermore, for any $p\in[1,\infty)$,
\begin{equation}\label{interpolate}
 \norm[\Lom p]{n(\cdot,t)}\le C_2^{\f1p} \kl{M(T')}^{1-\f1p} \qquad \text
{for all } t\in(0,T'). 
\end{equation}
In order to obtain an estimate for $M(T')$, 
for $t\in (0,T')$, we set $t_0:=(t-1)_+$ and 
represent $n$ according to 
\begin{align*}
  n(\cdot,t)
  &=e^{(t-t_0)\Delta}n(\cdot,t_0) 
  - \int_{t_0}^t e^{(t-s)\Delta}\nabla\cdot 
  \left(\chi \frac{n(\cdot,s)}{c(\cdot,s)}
  \nabla c(\cdot,s) + n(\cdot,s)u(\cdot,s) \right)\,ds. 
\end{align*}
If $t_0=0$ (that is, $t\le 1$), then 
\[
 \norm[\Lom\infty]{e^{(t-t_0)\Delta}n(\cdot,t_0)} = \norm[\Lom\infty]{e^{t\Delta} n_0} \le \norm[\Lom\infty]{n_0}, 
\]
whereas, if $t_0>0$ (i.e. $t>1$), then with the constant $C_4>0$ yielded by the semigroup estimate \cite[Lemma 1.3 i)]{win_aggregationvs} we have 
\[
 \norm[\Lom\infty]{e^{(t-t_0)\Delta}n(\cdot,t_0)} \le C_4\Big(1+(t-t_0)^{-\f N2}\Big) \norm[\Lom1]{n(\cdot,t_0)}=2C_2C_4, 
\]
because $t-t_0=1$. With $C_5:=\max\{\norm[\Lom\infty]{n_0}, 2C_2C_4\}$ and some $C_6>0$ obtained from the semigroup estimate in \cite[Lemma 2.1 iv)]{caolan16}, we have 
\begin{equation}\label{someestimate}
 \norm[\Lom\infty]{n(\cdot,t)}\le C_5 + C_6 \int_0^1 \kl{1+(t-s)^{-\f12-\f N{2r}}} \norm[\Lom r]{\left(\f n c \na c\right) (\cdot,s) - (nu)(\cdot,s)} ds
\end{equation}
for all $t\in(0,\tmax)$.
The definitions of $C_1$ and $C_3$ in \eqref{def:c1c2} and \eqref{def:c3} together with Hölder's inequality and \eqref{interpolate} imply 
\[
 \norm[\Lom r]{\kl{\f nc \na c}(\cdot,t)} \le C_1 \norm[\Lom{\f{rq}{q-r}}]{n(\cdot,t)}\norm[\Lom q]{\na c(\cdot,t)} \le C_1C_2^{\f1r-\f1q}C_3(M(T'))^{1+\f1q-\f1r} 
\]
for all $t\in(0,T')$ 
and with $C_6$ obtained from either Lemma \ref{lem;esti;u} or Lemma \ref{lem;esti;u;Navier}
\[
 \norm[\Lom r]{n(\cdot,t)u(\cdot,t)}\le C_2^{\f1r}C_6(M(T'))^{1-\f1r}, \qquad \text{for } t\in(0,T').
\]
Estimate \eqref{someestimate} is hence transformed into 
\[
 \norm[\Lom\infty]{n(\cdot,t)} \le C_4 + C_7 (M(T'))^{1+\f1q-\f1r} + C_7 (M(T'))^{1-\f1r} \quad \text{for all } t\in(0,T'),
\]
where $C_7=C_5\int_0^1 \sigma^{-\f12-\f N{2r}} d\sigma \max\{C_1C_2^{\f1r-\f1q}C_3, C_2^{\f1r}C_6\}$ is finite due to $r>N$. 
Accordingly, for any $T'\in(0,T)$ we have that 
\[
 M(T')\le \sup\left\{\xi\in\R\mid \xi\le C_4+C_7\xi^{1+\f1q-\f1r}+C_7\xi^{1-\f1r}\right\},
\]
which is a finite number due to $1+\f1q-\f1r<1$ and $1-\f1r<1$.
This concludes the proof. 
\end{proof}

%
%
%
%
\begin{proof}[{\bf Proof of Theorem \ref{mainthm}}]
With $\alpha_0\in(\f N4,\alpha]$ satisfying $\alpha_0<1-\f N2 \chi^2 +\f N4$, we may view $u_0$ as an element of $D(A^{\alpha_0})$ and apply Lemma \ref{lem;local existence} so as to obtain a solution that either exists globally or satisfies 
\begin{equation}\label{extend:lastproof}
 \lim_{t\to \tmax}
  \left(
  \lp{\infty}{n(\cdot,t)}
  + \|c(\cdot,t)\|_{W^{1,\vartheta}(\Omega)}
  +\lp{2}{A^{\alpha_0} u(\cdot,t)}
  \right)=\infty. 
\end{equation}
If $\tmax$ were finite, we could apply Lemma \ref{lem;esti;u} or Lemma \ref{lem;esti;u;Navier} and Lemma \ref{lem;Linfty;n} with $T=\tmax$ so as to see that $\lp{\infty}{n(\cdot,t)}+\lp{2}{A^{\alpha_0} u(\cdot,t)}$ were bounded on $(0,\tmax)$. Combining the boundedness of $n$ with Lemma \ref{lem;esti;nablav} and invoking Lemma \ref{lem;q;c}, again for $T=\tmax$, then would give rise to a contradiction to \eqref{extend:lastproof}. Therefore, $\tmax=\infty$.
\end{proof}

\section{Acknowledgement}
T.B. and J.L. acknowledge support of the {\em Deutsche Forschungsgemeinschaft}  within the project {\em Analysis of chemotactic cross-diffusion in complex frameworks}.
M.M. is funded by JSPS Research 
Fellowships for Young Scientists (No. 17J00101). 
A major part of this work was written while M.M. 
visited Universit\"at Paderborn under the support 
from Tokyo University of Science.


{\footnotesize

\def\cprime{$'$}

}
\end{document}